\documentclass[pdflatex,sn-mathphys-num]{sn-jnl}
\usepackage[utf8]{inputenc}
\usepackage[english]{babel}
\usepackage{cmap}
\usepackage{microtype}
\usepackage{amsmath,amsthm,amsfonts,amssymb}
\usepackage{mathtools}
\usepackage{xcolor}
\usepackage{xurl}
\hypersetup{hidelinks,pdftitle={Infinite series of Deza graphs with strongly regular children},pdfauthor={Mikhail P. Golubyatnikov}}
\usepackage{cleveref}
\theoremstyle{thmstyleone}
\newtheorem*{thm*}{Theorem}
\newtheorem*{lem*}{Lemma}
\newtheorem{thm}{Theorem}[section]

\newtheorem{lem}[thm]{Lemma}
\newtheorem{prop}[thm]{Proposition}
\theoremstyle{thmstylethree}

\theoremstyle{thmstyletwo}
\newtheorem{rem}{Remark}

\newcommand{\FF}{\mathbb{F}}

\usepackage{etoolbox}
\makeatletter
\patchcmd{\@maketitle}{\vskip21pt}{\vskip0pt}{}{\PackageError{manuscript}{Title spacing patch failed}{Check the supplied sn-jnl class.}}
\patchcmd{\@maketitle}{\vskip20pt}{\vskip0pt}{}{\PackageError{manuscript}{Title spacing patch failed}{Check the supplied sn-jnl class.}}
\makeatother

\begin{document}
\title[Deza graphs with strongly regular children]{Infinite series of Deza graphs with strongly regular children}
\author*[1,2,3]{\fnm{Mikhail P.} \sur{Golubyatnikov}}\email{mike\_ru1@mail.ru}
\affil[1]{\orgname{Krasovskii Institute of Mathematics and Mechanics, Ural Branch of the Russian Academy of Sciences}, \orgaddress{\city{Yekaterinburg}, \country{Russia}}}
\affil[2]{\orgname{Ural Mathematical Center}, \orgaddress{\city{Yekaterinburg}, \country{Russia}}}
\affil[3]{\orgname{Ural Federal University named after the First President of Russia B. N. Yeltsin}, \orgaddress{\city{Yekaterinburg}, \country{Russia}}}

\abstract{A Deza graph is a regular graph in which the number of common neighbours of two distinct vertices takes at most two values, regardless of adjacency. Its children are the graphs on the same vertex set in which adjacency is determined by these two common-neighbour counts. A Deza graph is called strongly Deza if both children are strongly regular. We construct an infinite family of edge-regular strongly Deza graphs using non-degenerate quadratic forms over the field with five elements. For every odd dimension greater than three and each of the two determinant square classes, we obtain a graph on the projective points represented by vectors of norm one. Its children are complementary strongly regular graphs with the parameters of the corresponding orthogonality graph on non-isotropic points and its complement. We determine the parameters by counting solutions to systems involving the associated bilinear form. These counts also yield symmetric association schemes over finite fields of odd characteristic. Further constructions include Deza graphs in dimension four, orthogonality graphs in odd dimensions, unions of relations in even dimensions over the field with nine elements, and an odd-dimensional family over the field with thirteen elements. We also give low-dimensional examples, including one whose children are a triangular graph and its complement.}
\keywords{Deza graph, strongly Deza graph, strongly regular graph, association scheme, quadratic form, finite field}
\pacs[MSC Classification]{05E30}
\maketitle

\vspace{-30pt}
{\small
\noindent\textbf{Acknowledgements}\par
\noindent This work was carried out at the Ural Mathematical Center with financial support from the Ministry of Science and Higher Education of the Russian Federation (Agreement No.~075-02-2026-737). The author thanks Natalia V. Maslova and V. V. Kabanov for helpful comments that improved the manuscript.\par}
\clearpage

	\section{Introduction}\label{sec:introduction}

		Strongly regular graphs are central objects in algebraic graph theory and combinatorial design theory.
		A $k$-regular graph on $n$ vertices is called \emph{strongly regular} with parameters $(n,k,\lambda,\mu)$ if every two adjacent vertices share exactly $\lambda$ common neighbours and every two non-adjacent vertices share exactly $\mu$ common neighbours.
		This notion was introduced by Bose~\cite{Bose_srg} and is closely connected with association schemes, finite geometry, coding theory, and group theory; see the monographs~\cite{BVM, BCN} for a comprehensive treatment.

		Deza graphs were introduced by Erickson, Fernando, Haemers, Hardy, and Hemmeter~\cite{Deza_1} as a natural generalisation of strongly regular graphs.
		A graph $\Gamma$ on $n$ vertices is a \emph{Deza graph} with parameters $(n,k,b,a)$ if it is $k$-regular and, for every pair of distinct vertices $u$ and $v$, the number of their common neighbours belongs to the set $\{a,b\}$, where $0 \le a \le b \le k$.
		The crucial difference from strongly regular graphs is that the value of $|\Gamma(u) \cap \Gamma(v)|$ need not depend on whether $u$ and $v$ are adjacent.
		In~\cite{Deza_1}, several constructions of Deza graphs were presented and the basic theory was developed, including the notion of \emph{children} of a Deza graph.
		Specifically, if $A$ is the adjacency matrix of a Deza graph $\Gamma$, then $A^2 = kI + aB_1 + bB_2$, 	where $B_1$ and $B_2$ are symmetric $(0,1)$-matrices with zero diagonal summing to $J - I$.
		The graphs $\Delta_1$ and $\Delta_2$ with adjacency matrices $B_1$ and $B_2$ are called the children of $\Gamma$.

		A Deza graph is called \emph{strictly Deza} if it has diameter~$2$ and is not strongly regular.
		The structure and classification of strictly Deza graphs have been investigated in a number of works; see also the survey by Goryainov and Shalaginov (\href{https://arxiv.org/abs/2103.00228}{arXiv:2103.00228}).
		Deza graphs with parameters $(n,k,k,a)$ were classified in~\cite{Deza_1}.
		For parameters $(n,k,k-1,a)$, let $\beta$ denote the number of vertices having $k-1$ common neighbours with a given vertex.
		Kabanov, Maslova, and Shalaginov~\cite{KMS} classified strictly Deza graphs with parameters $(n,k,k-1,a)$ and $\beta > 1$.
		Goryainov, Haemers, Kabanov, and Shalaginov~\cite{GHKS} completed the classification by characterising the remaining case $\beta = 1$ and showed that many such graphs arise from strongly regular graphs via the dual Seidel switching operation.
		The spectral theory of Deza graphs was developed by Akbari, Ghodrati, Hosseinzadeh, Kabanov, Konstantinova, and Shalaginov~\cite{Spectra_Deza}, who obtained relations between the eigenvalues of a Deza graph and those of its children.

		A Deza graph whose children are both strongly regular is called a \emph{strongly Deza graph}.
		This notion was studied systematically by Akbari, Haemers, Hosseinzadeh, Kabanov, Konstantinova, and Shalaginov~\cite{Spectra_strongly_Deza}, who gave a spectral characterisation of strongly Deza graphs, established relationships between their eigenvalues, and investigated those that are distance-regular.
		An important subclass of strongly Deza graphs is formed by divisible design graphs, introduced by Haemers, Kharaghani, and Meulenberg~\cite{DDG}, in which the vertex set admits a partition into classes with the property that the number of common neighbours of two vertices depends only on whether they belong to the same class.

		Methods for constructing strongly Deza graphs from strongly regular graphs by means of dual Seidel switching and its generalisations were described in~\cite{Deza_1} and further developed by Kabanov, Konstantinova, and Shalaginov~\cite{Kab}, who presented a general approach applicable when the initial strongly regular graph possesses a suitable involutory automorphism.
		Despite this progress, relatively few infinite families of strongly Deza graphs are known, especially those with explicitly described strongly regular children.

		Association schemes arising from finite orthogonal groups acting on non-isotropic points were studied by Bannai, Hao, and Song~\cite{BHS}.
		Adriaensen and De Boeck~\cite{ADB} give a recent treatment of association schemes and orthogonality graphs on anisotropic points of polar spaces.
		These works provide the geometric and algebraic context for the construction considered here.

		In this paper, we construct a new infinite family of strongly Deza graphs using association schemes arising from non-degenerate quadratic forms over finite fields of odd characteristic.
		More precisely, for every odd $n \ge 5$ we construct a symmetric association scheme on the set of projective points represented by vectors of norm~$1$ in $\FF_5^n$, and show that a certain relation of this scheme defines a Deza graph whose children are strongly regular graphs with the same parameters as the graphs $NO^{\varepsilon\perp}_n(5)$ and their complements, where $NO^{\varepsilon\perp}_n(q)$ denotes a class of strongly regular graphs defined on non-isotropic points of an orthogonal polar space~\cite{BVM, BCN}.

		Our main result is the following theorem.
		\begin{thm*}
		Let $n>3$ be an odd integer and let $\varepsilon\in\{-1,1\}$. Then there exists an edge-regular Deza graph with parameters
		\[
		\left(
		\frac{5^{n-1} + \varepsilon 5^{\frac{n-1}{2}}}{2},
		5^{n-2} - \varepsilon 5^{\frac{n-3}{2}},
		2\cdot5^{n-3} + \frac{1-3\varepsilon}{2}5^{\frac{n-3}{2}},
		2\cdot5^{n-3} - \frac{1+3\varepsilon}{2}5^{\frac{n-3}{2}}
		\right).
		\]

		Moreover, its two Deza children are strongly regular graphs with
		parameters
		\[
		\left(
		\frac{5^{n-1} + \varepsilon 5^{\frac{n-1}{2}}}{2},
		\frac{5^{n-2} - \varepsilon 5^{\frac{n-3}{2}}}{2},
		\frac{5^{n-3} + \varepsilon 5^{\frac{n-3}{2}}}{2},
		\frac{5^{n-3} - \varepsilon 5^{\frac{n-3}{2}}}{2}
		\right)
		\]
		and
		\[
		\left(
		\frac{5^{n-1} + \varepsilon 5^{\frac{n-1}{2}}}{2},\;
		2\cdot 5^{n-2} + 3\varepsilon 5^{\frac{n-3}{2}}-1,\;
		8\cdot 5^{n-3} + 3\varepsilon 5^{\frac{n-3}{2}}-2,\;
		8\cdot 5^{n-3} + 4\varepsilon 5^{\frac{n-3}{2}}
		\right),
		\]
		respectively.
		\end{thm*}

		The paper is organised as follows.
		In Section~\ref{sec:introduction}, we recall the necessary definitions and background on Deza graphs, association schemes, and quadratic forms over finite fields.
		In Section~\ref{sec:counting}, we determine the number of solutions to a system of equations involving a bilinear form and use this to construct a symmetric association scheme.
		In Section~\ref{sec:regularity}, we prove the strong regularity of the graph corresponding to a certain relation of the scheme.
		In Section~\ref{sec:families}, we give several families and examples of Deza graphs.
		In Section~\ref{sec:q5}, we establish the main result and compute the parameters of the graphs and their strongly regular children explicitly.

	\subsection{Deza graphs}
	Throughout, graphs are finite, simple, and undirected. Let $\Gamma = (V,E)$ be a graph.
	For an edge $\{u, v\}$, the vertices $u$ and $v$ are called \emph{adjacent}, which is denoted by $u \sim v$.
	The \emph{distance} between vertices $u$ and $v$ is defined as the length of a shortest path connecting them and is denoted by $d(u,v)$.
	A regular graph is called \emph{edge-regular} if every two adjacent vertices have the same number of common neighbours.
	The \emph{diameter} of a graph $\Gamma$ is the maximum distance between any two of its vertices.
	A subgraph $\Delta$ of a graph $\Gamma$ is called \emph{induced} if for any two vertices
	$u, v$ in $\Delta$, we have $u \sim v$ in $\Delta$ if and only if $u \sim v$ in $\Gamma$.
	For any vertex $v$, we define the \emph{neighbourhood} $\Gamma(v)$ and the \emph{second neighbourhood} $\Gamma_2(v)$ as the induced subgraphs on the sets $\{u \mid u \sim v\}$ and $\{u \mid d(v,u) = 2\}$, respectively.
	A graph is called \emph{$k$-regular} if each of its vertices has exactly $k$ neighbours.
	A graph is called \emph{strongly regular} with parameters $(n,k,\lambda,\mu)$ if
	\begin{enumerate}
		\item $\Gamma = (V,E)$, $|V| = n$, and $\Gamma$ is $k$-regular.
		\item For any two adjacent vertices $u, v$, $|\Gamma(v) \cap \Gamma(u)| = \lambda$.
		\item For any two distinct non-adjacent vertices $u, v$, $|\Gamma(v) \cap \Gamma(u)| = \mu$.
	\end{enumerate}

	Let $n$, $k$, $b$, $a$ be integers such that $0 \le a \le b \le k < n$.
	A graph $\Gamma$ is called a \emph{Deza graph} with parameters $(n,k,b,a)$ if it has $n$ vertices, is $k$-regular, and for any two distinct vertices $u, v$, the number $|\Gamma(v) \cap \Gamma(u)|$ belongs to $\{a, b\}$.
	The class of Deza graphs generalises the class of strongly regular graphs.
	Let $\Gamma$ be a Deza graph and $A$ be its adjacency matrix.
	Then there exist symmetric $(0,1)$-matrices $B_1, B_2$ with zero diagonal such that
	\[
	A^2 = kI + aB_1 + bB_2,
	\]
	where $I$ is the identity matrix and $B_1+B_2=J-I$, with $J$ the all-ones matrix.
	The matrices $B_1, B_2$ are the adjacency matrices of graphs $\Delta_1, \Delta_2$, called the children of the graph $\Gamma$ (see~\cite{Deza_1}).
	If the children $\Delta_1$ and $\Delta_2$ of a Deza graph $\Gamma$ are strongly regular, then $\Gamma$ is called a \emph{strongly Deza graph}.
	Methods for constructing strongly Deza graphs from strongly regular graphs using a suitable automorphism are described in~\cite{Deza_1,Kab}.
	\subsection{Association schemes}
	Let $X$ be a nonempty finite set, and let $R_0, R_1, \ldots, R_d$ be nonempty subsets
	of the Cartesian product $X \times X$ satisfying the following conditions:
	\begin{enumerate}
		\item $R_0 = \{(x,x) \,|\, x \in X\}$.
		\item $R_0 \sqcup R_1 \sqcup \cdots \sqcup R_d = X \times X$.
		\item $R_i^T = R_i$ for $0 \leq i \leq d$, where $R_i^T = \{(y,x) \,|\, (x,y) \in R_i\}$.
		\item There exist integers $p^h_{i,j}$ $(0 \leq h,i,j \leq d)$ such that for $(x,y) \in R_h$,
		\[
		p^h_{i,j} = |
		\{ z \in X \,|\, (x,z) \in R_i,\; (z,y) \in R_j \} |.
		\]
	\end{enumerate}
	Then the pair
	$
	\mathcal A=(X,\{R_i\}_{i=0}^d)
	$
	is called a \emph{symmetric association scheme}. The integers $p_{i,j}^h$ are called the
	\emph{intersection numbers} of $\mathcal A$ (see~\cite{BCN}).

	Symmetric association schemes can be used to construct graphs in a natural way, by taking
	unions of some of their relations. Let $A_i$ be the adjacency matrix of $R_i$. Namely, for a subset $F\subseteq\{1,2,\dots,d\}$, one
	considers the graph on $X$ whose adjacency matrix is $\sum_{f\in F} A_f$. The next theorem
	shows that the Deza property of this graph can be characterised entirely in terms of the
	intersection numbers of the underlying association scheme.

	\begin{thm*}[\cite{Deza_1}, Theorem~4.2]
		Let $(X,\{R_i\}_{i=0}^d)$ be a symmetric association scheme, and let $F\subseteq\{1,2,\dots,d\}$.
		Denote by $\Delta$ the graph with adjacency matrix
		$
		\sum_{f\in F}A_f.
		$
		Then $\Delta$ is a Deza graph if and only if the sum
		$
		\sum_{f,g\in F}p^k_{f,g}
		$
		takes at most two distinct values as $k$ ranges over $\{1,\dots,d\}$.
	\end{thm*}

	\subsection{Quadratic forms}
	\begin{prop}[see \cite{Grove}]
		Let $\FF_q = \mathrm{GF}(q)$ be a finite field of odd characteristic, and let $V = \FF_q^n$.
		Then every non-degenerate quadratic form on $V$ is equivalent to a diagonal form $x_1^2 + x_2^2 + \cdots + d x_n^2$
		with $d \in \FF_q^*$.
	\end{prop}
	We set $\det Q = d$; its square class is independent of the chosen diagonalisation.
	Throughout the paper, $Q$ is non-degenerate, $q$ is odd, and $V=\FF_q^n$.
	We write $\mu$ for the quadratic character of $\FF_q$, extended by $\mu(0)=0$.

	The symmetric bilinear form associated with $Q$ is defined by
	\[
	B(\mathbf{x}, \mathbf{y}) = \frac{1}{2} \left(Q(\mathbf{x}+\mathbf{y}) - Q(\mathbf{x}) - Q(\mathbf{y})\right).
	\]

	\begin{lem}[see \cite{Lidl}]\label{number_of_sols_0}
		Let $Q$ be a non-degenerate quadratic form over the field $\FF_q=\mathrm{GF}(q)$ in $n$ variables.
		Then
		\begin{enumerate}
			\item If $n$ is odd, then
			\[
			|\{\mathbf{v} \in \FF_q^n \,|\, Q(\mathbf{v}) = 0\}| = q^{n-1}.
			\]
			\item If $n$ is even, then
			\[
			|\{\mathbf{v} \in \FF_q^n \,|\, Q(\mathbf{v}) = 0\}|
			= q^{n-1} + \frac{q-1}{q} \,\mu (\det Q) \cdot (\mu(-1)q)^{\frac{n}{2}},
			\]
			where $\mu$ denotes the quadratic character of the field $\FF_q$.
		\end{enumerate}
	\end{lem}

	\begin{lem}[see \cite{Lidl}]\label{number_of_sols_a}
		Let $Q$ be a non-degenerate quadratic form over the field $\FF_q=\mathrm{GF}(q)$ in $n$ variables, and let $a \in \FF_q^*$.
		Then
		\begin{enumerate}
			\item If $n$ is odd, then
			\[
			|\{\mathbf{v} \in \FF_q^n \,|\, Q(\mathbf{v}) = a\}| = q^{n-1} + \mu(a)\, \mu((-1)^{\frac{n-1}{2}}\det Q)\, q^{\frac{n-1}{2}}.
			\]
			\item If $n$ is even, then
			\[
			|\{\mathbf{v} \in \FF_q^n \,|\, Q(\mathbf{v}) = a\}|
			= q^{n-1} - \mu((-1)^{\frac{n}{2}} \det Q)\, q^{\frac{n}{2}-1},
			\]
			where $\mu$ denotes the quadratic character of the field $\FF_q$.
		\end{enumerate}
	\end{lem}

	\section{Counting solutions to systems involving a bilinear form}\label{sec:counting}
	\begin{lem}[see~\cite{Lam}, Proposition~1.12]\label{invariant_lemma}
		Let $\mathbf{v}_1, \mathbf{v}_2, \dots, \mathbf{v}_n$ and $\mathbf{w}_1, \mathbf{w}_2, \dots, \mathbf{w}_n$ be anisotropic orthogonal bases of $V$.
		Then
		\[
		\mu(Q(\mathbf{v}_1) Q(\mathbf{v}_2) \cdots Q(\mathbf{v}_n)) = \mu(Q(\mathbf{w}_1) Q(\mathbf{w}_2) \cdots Q(\mathbf{w}_n)).
		\]
	\end{lem}
	Write $\mu(V) = \mu(Q(\mathbf{v}_1) Q(\mathbf{v}_2) \cdots Q(\mathbf{v}_n))$.
	For a degenerate subspace $W$, set $\mu(W)=0$.
	If $V = W \oplus U$ is an orthogonal direct sum, then $\mu(V) = \mu(W)\mu(U)$.

	\begin{lem}\label{invariant_simple_lemma}
		Let $\mathbf{v}$ and $\mathbf{u}$ be linearly independent anisotropic vectors in $V$.
		Then
		\[\mu(\langle \mathbf{v}, \mathbf{u}  \rangle) = \mu\left(  Q(\mathbf{v})Q(\mathbf{u}) - B(\mathbf{v}, \mathbf{u})^2 \right).
		\]
	\end{lem}
	\begin{proof}
		Apply the Gram--Schmidt process to the vectors $(\mathbf{v}, \mathbf{u})$. Then
		\[
		\langle \mathbf{v}, \mathbf{u}  \rangle = \left\langle \mathbf{v}, \mathbf{u} - \frac{ B(\mathbf{v}, \mathbf{u}) }{Q(\mathbf{v})}\mathbf{v}  \right\rangle,
		\]
		\[
		Q\left( \mathbf{u} - \frac{ B(\mathbf{v}, \mathbf{u}) }{Q(\mathbf{v})}\mathbf{v}\right) = Q(\mathbf{u}) -
		\frac{B(\mathbf{v}, \mathbf{u})^2}{Q(\mathbf{v})}.
		\]
		If the second basis vector is isotropic, the plane is degenerate and both sides of the claimed identity are zero.
		Otherwise, Lemma~\ref{invariant_lemma} gives
		\[
		\mu(\langle \mathbf{v}, \mathbf{u}  \rangle) = \mu\left(  Q(\mathbf{v}) \left( Q(\mathbf{u}) -
		\frac{B(\mathbf{v}, \mathbf{u})^2}{Q(\mathbf{v})} \right) \right)   =\mu\left(  Q(\mathbf{v})Q(\mathbf{u}) - B(\mathbf{v}, \mathbf{u})^2 \right).
		\]
	\end{proof}

	A pair $(\mathbf{x},\mathbf{y})$ in $V$ is called a \emph{hyperbolic pair} if
	$Q(\mathbf{x})=Q(\mathbf{y})=0$ and $B(\mathbf{x},\mathbf{y})=1$.

	\begin{lem}\label{invariant_hyperpane_lemma}
		Let $(\mathbf{x}, \mathbf{y})$ be a hyperbolic pair in $V$. Then $\mu(\langle \mathbf{x}, \mathbf{y} \rangle) = \mu(-1)$.
	\end{lem}
	\begin{proof}
		Consider $\mathbf{a} = \mathbf{x} + \mathbf{y}$ and $\mathbf{b} = \mathbf{x} - \mathbf{y}$.
		Note that $B(\mathbf{a}, \mathbf{b}) = B(\mathbf{x}, \mathbf{x}) - B(\mathbf{y}, \mathbf{y}) = 0$.
		Furthermore,
		\begin{align*}
			Q(\mathbf{a}) &= Q(\mathbf{x}) + 2B(\mathbf{x},\mathbf{y}) + Q(\mathbf{y}) = 2, \\
			Q(\mathbf{b}) &= Q(\mathbf{x}) - 2B(\mathbf{x},\mathbf{y}) + Q(\mathbf{y}) = -2.
		\end{align*}
		Hence $(\mathbf{a}, \mathbf{b})$ is an orthogonal basis of $\langle \mathbf{x}, \mathbf{y} \rangle$, and by Lemma~\ref{invariant_lemma} we have
		\[\mu(\langle \mathbf{x}, \mathbf{y} \rangle) = \mu(Q(\mathbf{a})Q(\mathbf{b})) = \mu(-4) = \mu(-1).\]
	\end{proof}

	\begin{thm}\label{number_of_sols_K_thm}
		Let $S = \{\mathbf{v} \in \FF_q^n \,|\, Q(\mathbf{v}) = 1\}$.
		Let $a \in \FF_q$ and $\mathbf{v} \in S$; the unknown is $\mathbf{x} \in V$.
		Denote $\mathcal{N}({n, \det Q, q, a}) = |\{\mathbf{v} \in \FF_q^n \,|\, Q(\mathbf{v}) = a\}|$.
		For dimension zero, use the convention $\mathcal N(0,d,q,a)=1$ if $a=0$ and $0$ otherwise.
		Consider the system of equations
		\begin{equation}\label{Ka-system}
			\begin{cases}
				B(\mathbf{x}, \mathbf{x}) = 1, \\
				B(\mathbf{v}, \mathbf{x}) = a.
			\end{cases}
		\end{equation}
		Then the number of solutions to the system~\eqref{Ka-system} depends only on $a$ and is equal to
		\[
		K_{a} = \mathcal{N}(n-1, \det Q, q, 1-a^2).
		\]
	\end{thm}
	\begin{proof}
		Since $\mathbf{v}$ is anisotropic, we have a decomposition $V = \langle \mathbf{v} \rangle \oplus W$, where $W = \langle \mathbf{v} \rangle^\perp$.
		Let $\mathbf{x}$ be a solution of system~\eqref{Ka-system}, and let $\mathbf{x} = \alpha \mathbf{v} + \mathbf{w}$.
		Then system~\eqref{Ka-system} takes the form
		\[
		\begin{cases}
			\alpha^2 + Q(\mathbf{w}) = 1, \\
			\alpha = a.
		\end{cases}
		\]
		Thus any solution of system~\eqref{Ka-system} corresponds to a solution of
		$Q(\mathbf{w}) = 1 - a^2$ in $W$.
		By Lemma~\ref{invariant_lemma} we have
		$\mu(W) = \mu(V)\mu(\langle \mathbf{v} \rangle) = \mu(\det Q)$.
	\end{proof}

	For $\mathbf u,\mathbf v\in V$ and $a,b\in\FF_q$, consider the following system in the unknown $\mathbf x\in V$:

	\begin{equation}\label{Pvu-system}
		\begin{cases}
			B(\mathbf{x}, \mathbf{x}) = 1, \\
			B(\mathbf{v}, \mathbf{x}) = a, \\
			B(\mathbf{u}, \mathbf{x}) = b.
		\end{cases}
	\end{equation}
	We denote the solution set by $\mathcal{P}_{a,b}(\mathbf{v},\mathbf{u})$ and its cardinality by $P_{a,b}(\mathbf{v},\mathbf{u})$.
	\begin{thm}\label{number_of_sols_thm}
		Let $S = \{\mathbf{v} \in \FF_q^n \,|\, Q(\mathbf{v}) = 1\}$.
		Assume $n\ge3$. Let $a, b \in \FF_q$ and $\mathbf{u}, \mathbf{v} \in S$, with $\mathbf{u}$ and $\mathbf{v}$ linearly independent.
		Denote
		\[
		\mathcal{N}({n, \det Q, q, a}) = |\{\mathbf{v} \in \FF_q^n \,|\, Q(\mathbf{v}) = a\}|,
		\]
		$c=B(\mathbf{v},\mathbf{u})$ and $\delta = 1 - c^2$. If $\delta\ne0$, set
		\[
		\tau = \frac{a^2 - 2ab B(\mathbf{v}, \mathbf{u}) + b^2}{\delta}.
		\]
		Then the number of solutions of system~\eqref{Pvu-system} is given by
		\[
		P_{a, b}(\mathbf{v}, \mathbf{u}) = \begin{cases}
			\mathcal{N}(n-2, \delta \det Q, q, 1 - \tau) & \text{ if } \delta \ne 0, \\
			q^{n-3} & \text{ if } \delta = 0 \text{ and } b \ne ca, \\
			q\mathcal{N}(n-3, -\det Q, q, 1 - a^2) & \text{ if } n > 3, \delta = 0 \text{ and } b = ca, \\
			q & \text{ if } n = 3, \delta = 0, b = ca \text{ and } a^2 = 1, \\
			0 & \text{ if } n = 3, \delta = 0, b = ca
			\text{ and } a^2 \ne 1. 
\end{cases}
		\]
		In particular, the number of solutions depends only on $a, b$, and $B(\mathbf{v}, \mathbf{u})$.
	\end{thm}
	\begin{proof}
		Let $W = \langle \mathbf{v}, \mathbf{u} \rangle^\perp$.
		Consider $\mathbf{w} \in W \cap \langle \mathbf{v}, \mathbf{u} \rangle$.
		Let $\mathbf{w} = \alpha \mathbf{v} + \beta \mathbf{u}$. Then
		\[
		\begin{cases}
			0 = B(\mathbf{v}, \mathbf{w}) = \alpha Q(\mathbf{v}) + \beta B(\mathbf{v}, \mathbf{u}) = \alpha + \beta B(\mathbf{v}, \mathbf{u}),\\
			0 = B(\mathbf{u}, \mathbf{w}) = \beta Q(\mathbf{u}) + \alpha B(\mathbf{v}, \mathbf{u}) = \beta + \alpha B(\mathbf{v}, \mathbf{u}).
		\end{cases}
		\]
		If $\mathbf{w}\ne0$, this implies $B(\mathbf{v}, \mathbf{u})^2 = 1$.

		If $B(\mathbf{v}, \mathbf{u}) \ne \pm 1$, then $W \cap \langle \mathbf{v}, \mathbf{u} \rangle = \{0\}$, and we have a decomposition
		$V = \langle \mathbf{v}, \mathbf{u} \rangle \oplus W$.
		For $\mathbf{x} \in \mathcal{P}_{a,b}(\mathbf{v}, \mathbf{u})$, let
		\[\mathbf{x} = \alpha \mathbf{v} + \beta \mathbf{u} + \mathbf{w}.\]
		We obtain
		\[
		\begin{cases}
			\alpha^2 Q(\mathbf{v}) + 2 \alpha \beta B(\mathbf{v}, \mathbf{u}) + \beta^2 Q(\mathbf{u}) + Q(\mathbf{w}) = 1,\\
			\alpha Q(\mathbf{v}) + \beta B(\mathbf{v}, \mathbf{u}) = a, \\
			\beta Q(\mathbf{u}) + \alpha B(\mathbf{v}, \mathbf{u}) = b.
		\end{cases}
		\]
		Denote
		\[
		\tau = \alpha^2 Q(\mathbf{v}) + 2 \alpha \beta B(\mathbf{v}, \mathbf{u}) + \beta^2 Q(\mathbf{u}) = \frac{ a^2 Q(\mathbf{u}) - 2ab B(\mathbf{v}, \mathbf{u}) + b^2 Q(\mathbf{v}) }{\delta}.
		\]

		The number of solutions of system~\eqref{Pvu-system} coincides with the number of solutions to the equation $Q(\mathbf{w}) = 1 - \tau$ in $W$.
		By Lemmas~\ref{invariant_lemma} and \ref{invariant_simple_lemma} we have
		\[
		\mu(W) = \mu(V)\mu(\langle \mathbf{v}, \mathbf{u} \rangle) = \mu(\det Q)\mu(\delta).
		\]
		Hence the number of solutions of system~\eqref{Pvu-system} equals $\mathcal{N}(n-2, \delta \det Q, q, 1 - \tau)$.
		Now let $\delta = 0$, so $c\in\{-1,1\}$.
		Replacing $(\mathbf{u},b)$ by $(c\mathbf{u},cb)$ leaves the solution set unchanged and makes $B(\mathbf{v},c\mathbf{u})=1$.
		Thus we may work with $c=1$; in the original variables the condition $a=b$ below is $b=ca$.
		In this case
		$W \cap \langle \mathbf{v}, \mathbf{u} \rangle = \langle \mathbf{v} - \mathbf{u} \rangle$.
		The space $\langle\mathbf{v}\rangle^\perp$ is non-degenerate and contains the nonzero isotropic vector $\mathbf{v}-\mathbf{u}$.
		Choose $\mathbf{y}\in\langle\mathbf{v}\rangle^\perp$ such that $(\mathbf{v} - \mathbf{u}, \mathbf{y})$ is a hyperbolic pair.
		We then have an orthogonal decomposition (see~\cite[Proposition~4.11]{Grove})
		\[
		V = \langle \mathbf{v} \rangle \oplus \langle \mathbf{v} - \mathbf{u}, \mathbf{y} \rangle \oplus W',
		\]
		where $W' \subset W$.
		Let $\mathbf{x} \in \mathcal{P}_{a,b}(\mathbf{v}, \mathbf{u})$. Then
		\[
		\mathbf{x} = \alpha \mathbf{v} + \beta (\mathbf{v} - \mathbf{u}) + \gamma \mathbf{y} + \mathbf{w'},
		\]
		and
		\[
		\begin{cases}
			\alpha + \gamma B(\mathbf{v}, \mathbf{y}) = a, \\
			\alpha + \gamma B(\mathbf{u}, \mathbf{y}) = b.
		\end{cases}
		\]
		Since $B(\mathbf{v}-\mathbf{u}, \mathbf{y})=1$, we have $B(\mathbf{v}, \mathbf{y}) = 1 + B(\mathbf{u}, \mathbf{y})$, and thus
		\[
		\begin{cases}
			\alpha =(b - a)B(\mathbf{u}, \mathbf{y}) + b,\\
			\gamma = a-b.
		\end{cases}
		\]

		\[
		Q(\mathbf{x}) = \alpha^2 + Q(\mathbf{w'}) + 2 \gamma (\alpha B(\mathbf{v}, \mathbf{y}) + B(\mathbf{y}, \mathbf{w'})) + 2\beta\gamma.
		\]

		If $a = b$, then $\gamma = 0$ and $\alpha = b$, and $Q(\mathbf{x}) = \alpha^2 + Q(\mathbf{w'})$.
		Hence any solution of system~\eqref{Pvu-system}
		is obtained from $\mathbf{w'} \in W'$ satisfying the equation $Q(\mathbf{w'}) = 1 - a^2$ and an arbitrary choice of the parameter $\beta \in \FF_q$.
		If $n>3$, Lemmas~\ref{invariant_lemma} and \ref{invariant_hyperpane_lemma} give
		\[
		\mu(W') = \mu(V)\mu(\langle \mathbf{v} - \mathbf{u}, \mathbf{y} \rangle) = \mu(-\det Q).
		\]
		For $n>3$, the number of solutions is therefore $q\mathcal{N}(n-3, -\det Q, q, 1 - a^2)$.
		If $n=3$, then $W'=\{0\}$, so there are $q$ solutions when $a^2=1$ and no solutions otherwise.

		If $a \ne b$, then $\gamma \ne 0$. Let $\mathbf{w'} \in W'$ be an arbitrary vector.
		The condition $Q(\mathbf{x})=1$ uniquely determines $\beta\in\FF_q$ as
		\[
		\beta = \frac{1 - (\alpha^2 + Q(\mathbf{w'}) + 2 \gamma (\alpha B(\mathbf{v}, \mathbf{y}) + B(\mathbf{y}, \mathbf{w'})))}{2\gamma}.
		\]
		Thus, for any vector $\mathbf{w'} \in W'$, the parameters $\alpha, \beta, \gamma$ are uniquely determined, and the number of solutions equals $|W'|
		= q^{n-3}$.
	\end{proof}

	\subsection{Construction of an association scheme}
	Let $q$ be a power of an odd prime and let $n\ge3$.
	Let $S = \{\mathbf{v} \in \FF_q^n \,|\, Q(\mathbf{v}) = 1\}$.
	Define
	\[S^+ = \{\{\mathbf{v}, -\mathbf{v}\} \,|\, \mathbf{v} \in S\},\]
	\[\widetilde{\mathbb{F}}_q = \{\{x, -x\} \,|\, x \in \FF_q\}.\]

	In what follows, we use a representative $\mathbf v\in S$ to denote the pair $\{\mathbf v,-\mathbf v\}\in S^+$ whenever no ambiguity arises.
	Similarly, we use a representative $a\in\FF_q$ to denote the class $\{a,-a\}\in\widetilde{\mathbb{F}}_q$.
	The notation $x = \pm a$ means $x = a$ or $x = -a$.
	We define relations on $S^+$ in two cases.

	Assume $n > 3$, or $n=3$ and $\mu(-\det Q)=1$:
	\begin{align*}
		{\cal R}_\Delta &= \{(\mathbf{v}, \mathbf{v}) \,|\, \mathbf{v}\in S^+\}, \\
		{\cal R}_a &= \{(\mathbf{v}, \mathbf{u}) \,|\, \mathbf{v} \ne \mathbf{u}, B(\mathbf{v}, \mathbf{u}) = \pm a\}, \quad a \in \widetilde{\mathbb{F}}_q.
	\end{align*}

	Assume $n = 3$ and $\mu(-\det Q)=-1$:
	\begin{align*}
		{\cal R}_\Delta &= \{(\mathbf{v}, \mathbf{v}) \,|\, \mathbf{v}\in S^+\}, \\
		{\cal R}_a &= \{(\mathbf{v}, \mathbf{u}) \,|\, B(\mathbf{v}, \mathbf{u}) = \pm a\}, \quad a \in \widetilde{\mathbb{F}}_q \setminus\{1\}.
	\end{align*}

	\begin{thm*}
		If $n > 3$, or $n=3$ and $\mu(-\det Q)=1$, then
		$\mathcal{A} = (S^+, \{ {\cal R}_\Delta \} \cup \{ {\cal R}_a \,|\, a \in \widetilde{\mathbb{F}}_q\})$ is an association scheme with $\frac{q+1}{2}$ classes.
		If $n = 3$ and $\mu(-\det Q)=-1$, then
		$\mathcal{A} = (S^+, \{ {\cal R}_\Delta \} \cup \{ {\cal R}_a \,|\, a \in \widetilde{\mathbb{F}}_q \setminus \{1\}\})$ is an association scheme with $\frac{q-1}{2}$ classes.
	\end{thm*}
	\begin{proof}
	Let $\mathcal I$ be the set of non-diagonal indices retained above.
	By Theorem~\ref{number_of_sols_K_thm}, the valency of $\mathcal R_1$, before omitting an empty relation, is
	$\mathcal N(n-1,\det Q,q,0)-1$.
	For $n=3$ this equals $(q-1)(1+\mu(-\det Q))$, and for $n>3$ it is positive.
	The valencies of the other relations are positive by Lemma~\ref{number_of_sols_a}.
	Thus the retained relations are nonempty and partition $S^+\times S^+$, and they are symmetric.
	For distinct elements $\mathbf{v}, \mathbf{u} \in S^+$ and $a,b \in \mathcal I$, consider
	\[
	p_{a,b}(\mathbf{v}, \mathbf{u}) = \{\mathbf{x} \in S^+ \,|\, (\mathbf{v}, \mathbf{x}) \in {\cal R}_a,\;
	(\mathbf{x}, \mathbf{u}) \in {\cal R}_b \},
	\]
	or equivalently,
	\[
	p_{a,b}(\mathbf{v}, \mathbf{u}) = \{\mathbf{x} \in S^+ \setminus \{\mathbf{v}, \mathbf{u}\} \,|\, B(\mathbf{v}, \mathbf{x}) = \pm a,\, B(\mathbf{u}, \mathbf{x}) = \pm b \}.
	\]

	For fixed representatives, write $\alpha=B(\mathbf{v},\mathbf{u})$ and abbreviate the number of solutions as $P_{a,b}(\alpha)$. Then
	\[
	|p_{a,b}(\mathbf{v}, \mathbf{u})| =
	\frac12\sum_{(x,y)\in S(a,b)} P_{x,y}(\alpha)
	\;-\;\delta_{a,1}\,\delta_{\alpha,b}\;-\;\delta_{b,1}\,\delta_{\alpha,a}.
	\]
	Here $(\mathbf{v}, \mathbf{u}) \in {\cal R}_\alpha$,
	\[
	S(a,b)=\{(\varepsilon_1 a,\varepsilon_2 b)\;\mid\;\varepsilon_1,\varepsilon_2\in\{-1, 1\}\}
	\]
	and the Kronecker deltas compare indices in $\widetilde{\mathbb F}_q$, so equality of non-diagonal indices is understood up to sign.

	By Theorem~\ref{number_of_sols_thm}, for fixed field elements $a,b\in\FF_q$ and $\mathbf{v}, \mathbf{u} \in S$, the number $P_{a,b}(\mathbf{v}, \mathbf{u})$ depends only on $B(\mathbf{v}, \mathbf{u})$.
	Note that $\mathbf{v} \in \mathcal{P}_{a,b}(\mathbf{v}, \mathbf{u})$ if and only if $a = 1$ and $B(\mathbf{v}, \mathbf{u}) = b$.
	Similarly, $\mathbf{u} \in \mathcal{P}_{a,b}(\mathbf{v}, \mathbf{u})$ if and only if $b = 1$ and $B(\mathbf{v}, \mathbf{u}) = a$.
	We extend the definition of $p_{a,b}(\mathbf{v}, \mathbf{u})$ to $\mathbf{v}, \mathbf{u} \in S^+$ and $a,b \in \mathcal I \cup \{\Delta\}$.
	If $\mathbf{v} = \mathbf{u}$ and $a, b \ne \Delta$, the set $p_{a,b}(\mathbf{v}, \mathbf{v})$ is nonempty only if $a = b$, and in this case it can similarly be expressed in terms of the number of solutions of the system
	\begin{equation*}
		\begin{cases}
			B(\mathbf{x}, \mathbf{x}) = 1, \\
			B(\mathbf{v}, \mathbf{x}) = a, 
\end{cases}
	\end{equation*}
	whose number of solutions, by Theorem~\ref{number_of_sols_K_thm}, depends only on the value of $a$.
	Let $\mathbf{v} \ne \mathbf{u}$ and $b\in\mathcal I$, and consider
	\[
	p_{\Delta,b}(\mathbf{v}, \mathbf{u}) = \{\mathbf{x} \in S^+ \,|\, (\mathbf{v}, \mathbf{x}) \in {\cal R}_\Delta,\;
	(\mathbf{x}, \mathbf{u}) \in {\cal R}_b \}.
	\]
	Then
	\[
	|p_{\Delta, b}(\mathbf{v}, \mathbf{u})| = \begin{cases}
		0 & \text{ if } B(\mathbf{v}, \mathbf{u}) \ne \pm b, \\
		1 & \text{ if } B(\mathbf{v}, \mathbf{u}) = \pm b.
	\end{cases}
	\]
	The cases involving $\Delta$ and coincident vertices are immediate from the diagonal relation.
	Thus, the numbers $|p_{a,b}(\mathbf{v},\mathbf{u})|$ depend only on the value of $B(\mathbf{v},\mathbf{u})$ in $\FF_q$, or on whether $\mathbf{v}=\mathbf{u}$. Moreover, since
	\[
	P_{a,b}(-\mathbf{v},\mathbf{u}) = P_{-a,b}(\mathbf{v},\mathbf{u})
	\quad\text{and}\quad
	P_{a,b}(\mathbf{v},-\mathbf{u}) = P_{a,-b}(\mathbf{v},\mathbf{u}),
	\]
	the number $|p_{a,b}(\mathbf{v},\mathbf{u})|$ depends only on the value of $B(\mathbf{v},\mathbf{u})$ 	in $\widetilde{\mathbb{F}}_q$, or on whether $\mathbf{v}=\mathbf{u}$; that is, only on the relation $\mathcal{R}_{\alpha}$ containing the pair $(\mathbf{v},\mathbf{u})$.
	\end{proof}

	If $(\mathbf{v}, \mathbf{u}) \in {\cal R}_\alpha$, we denote the value $|p_{a, b}(\mathbf{v}, \mathbf{u})|$ by $p_{a, b}^\alpha$ or $p_{a, b}(\alpha)$ for
	$a, b, \alpha \in \mathcal I \cup \{\Delta\}$.
	The adjacency matrix of the relation ${\cal R}_\alpha$ is denoted by $A_\alpha$.
\section{Strong regularity of the graph with adjacency matrix \texorpdfstring{$A_1$}{A1}}\label{sec:regularity}
Assume that $n>3$ is odd, or that $n=3$ and $\mu(-\det Q)=1$.

	Consider the intersection number $p_{1, 1}^\alpha$.
	Choose representatives with $B(\mathbf{v},\mathbf{u})=\alpha$ for each non-diagonal relation.
	The following equalities hold, with the first applying when $\alpha\notin\{1,\Delta\}$:

	\begin{align*}
		p_{1, 1}^\alpha &= 2P_{1, 1}(\mathbf{v}, \mathbf{u}), \\
		p_{1, 1}^\Delta &= \mathcal{N}(n-1, \det Q, q, 0)-1, \\
		p_{1, 1}^1      &= P_{1, 1}(\mathbf{v}, \mathbf{u}) + P_{1, -1}(\mathbf{v}, \mathbf{u}) - 2 = q\mathcal{N}(n-3, -\det Q, q, 0) + q^{n-3} - 2.
	\end{align*}

	Let $\alpha \notin \{\pm 1, \Delta\}$.
	Then
	\begin{align*}
		P_{1, 1}(\mathbf{v}, \mathbf{u}) = \mathcal{N}(n-2, \delta \det Q, q, 1 - \tau),
	\end{align*}
	where $\delta = 1 - \alpha^2$ and
	\[\tau = \frac{2}{1 + \alpha}.\]
	For odd $n$, we have
	\[
	\mathcal{N}(n-1, \delta \det Q, q, 0) = q^{n-2} + \frac{q-1}{q} \mu (\delta \det Q) \cdot (\mu(-1)q)^\frac{n-1}{2},
	\]

	\begin{align*}
		P_{1, 1}(\mathbf{v}, \mathbf{u}) &= \mathcal{N}(n-2, \delta \det Q, q, 1 - \tau) \\
		&= q^{n-3} + \mu((-1)^\frac{n-3}{2} \delta (1 - \tau) \det Q) q^\frac{n-3}{2} \\
		&= q^{n-3} + \mu(-(-1)^\frac{n-3}{2} (\alpha-1)^2 \det Q) q^\frac{n-3}{2} \\
		&= q^{n-3} + \mu(-(-1)^\frac{n-3}{2} \det Q) q^\frac{n-3}{2} \\
		&= q^{2m-2} + \varepsilon q^{m - 1},
	\end{align*}
	where $\varepsilon = \mu(- (-1)^{m - 1} \det Q)$ and $n = 2m + 1$.
	Hence
	\[
	p_{1, 1}^\alpha =
	\begin{cases}
		2q^{m - 1}(q^{m - 1} + \varepsilon)  & \text{ if } \alpha \notin \{1, \Delta\}, \\
		q^{2m-1} + \varepsilon (q-1)q^{m-1} - 1 & \text{ if } \alpha = \Delta, \\
		2(q^{2m - 2} - 1) + \varepsilon (q-1)q^{m - 1} & \text{ if } \alpha = 1. 
\end{cases}
	\]

	The graph with adjacency matrix $A_1$ has the same parameters as $NO^\varepsilon_{2m+1}(q)$.

	We next consider special cases for small fields. In Section~\ref{sec:q5}, we study a three-class scheme in which $A_2$ defines a Deza graph with strongly regular children corresponding to $A_0$ and $A_1+A_2$.
	\section{Families of Deza graphs}\label{sec:families}
	\begin{thm}\label{thm:families}
		The following statements hold:
		\begin{enumerate}
			\item[(1)] Let $n \ge 4$ be even, and $q = 9$.
			Then $A_0 + A_a$ is the adjacency matrix of a Deza graph for each $a\in\FF_9\setminus\FF_3$.

			\item[(2)] Let $n = 4$, and $q \ge 5$. Then the graph corresponding to relation $A_1$ is a Deza graph.
			\item[(3)] Let $n \ge 5$ be odd, and $q \ge 5$. Then the class $A_0$ corresponds to a Deza graph.

			\item[(4)] Let $n \ge 5$ be odd and $q = 13$.
			Then the graph corresponding to $A_0 + A_3 + A_4 + A_5$ is a Deza graph.

			\item[(5)] Let $n = 3$, and $q \ge 5$. Then $A_0$ is the adjacency matrix of a Deza graph whose common-neighbour counts belong to $\{0,1\}$.

			\item[(6)] Let $n = 4$, $q = 5$ and $\mu(\det Q) = -1$.
			Then $A_0 + A_1$ corresponds to a Deza graph with parameters
			$(65, 34, 18, 15)$.
			Moreover, the scheme $\mathcal{A}$ is $P$-polynomial with ordering $(\Delta, 0, 2, 1)$, and $A_0$ corresponds to a distance-regular graph with intersection array $\{10, 6, 4;
			1, 2, 5\}$.

			\item[(7)] Let $n = 3$, $q = 9$ and $\mu(\det Q) = 1$.
			If $\FF_9 = \FF_3(\theta)$ with $\theta^2 - \theta - 1 = 0$, then the complement of the graph corresponding to $A_{\theta + 1}$ is a Deza graph with parameters $(45, 36, 29, 28)$ and strongly regular children $T(10)$ and $\overline{T(10)}$, where $T(10)$ is the triangular graph on the $2$-subsets of a $10$-element set.
			\item[(8)] Let $n = 3$, $q = 13$ and $\mu(\det Q) = 1$.
			Then the graph corresponding to $A_0 + A_3 + A_4 + A_5$ is a Deza graph with parameters $(91, 42, 21, 17)$.

			\item[(9)] Let $n = 3$, $q = 13$ and $\mu(\det Q) = -1$.
			Then $A_0 + A_2 + A_3 + A_4$ corresponds to a Deza graph with parameters $(78, 49, 32, 28)$.

			\item[(10)] Let $n = 3$, $q = 17$ and $\mu(\det Q) = -1$.
			Then $A_4 + A_6$ corresponds to a Deza graph with parameters $(136, 36, 12, 8)$.

            \item[(11)] Let $n=3$, $q=5$ and $\mu(\det Q)=1$.
            Then $A_2$ is the adjacency matrix of a strongly Deza graph with parameters
            $(15,4,1,0)$ and diameter~$3$. Its children are $5K_3$ and its complement.

            \item[(12)] Let $n=3$, $q=7$ and $\mu(\det Q)=1$.
            Then $A_2$ and $A_0+A_3$ are the adjacency matrices of complementary
            strictly Deza graphs with parameters $(21,8,4,2)$ and $(21,12,7,5)$,
            respectively.

            \item[(13)] Let $n=4$, $q=7$ and $\mu(\det Q)=1$.
            Then $A_0+A_1+A_2$ is the adjacency matrix of a strictly Deza graph
            with parameters $(168,111,75,70)$.

            \item[(14)] Let $n=3$, $q=9$ and $\mu(\det Q)=-1$.
            Write $\FF_9=\FF_3(\theta)$, where $\theta^2-\theta-1=0$.
            Then $A_{\theta+1}$ and $A_0+A_\theta+A_{\theta-1}$ are the adjacency
            matrices of complementary strongly regular graphs with parameters
            $(36,10,4,2)$ and $(36,25,16,20)$, respectively, in the order
            $(v,k,\lambda,\mu)$.

		\end{enumerate}
	\end{thm}
	\begin{proof}
		Denote
		\[
		\varepsilon =
		\begin{cases}
			\mu((-1)^\frac{n-1}{2}\det Q), & n\text{ odd,} \\
			\mu((-1)^\frac{n}{2} \det Q), & n\text{ even}.
		\end{cases}
		\]

		Statements (6)--(14) are verified directly using SageMath~\cite{sage}.

		Statement~(5) follows from the fact that for linearly independent anisotropic vectors $\mathbf{v}, \mathbf{u} \in \FF_q^3$, we have $\dim \langle \mathbf{v}, \mathbf{u} \rangle ^\perp = 1$.

		We illustrate statement~(5) with some small graphs. Let $n = 3$, $q \ge 5$ and $\Gamma_{\varepsilon, q}$ be a graph with adjacency matrix $A_0$.

		For $q = 5$, we have $\Gamma_{+, 5} = 5K_3$ and $\Gamma_{-, 5}$ is the Petersen graph.

		For $q = 7$, the graph $\Gamma_{-, 7}$ is a distance-regular graph with intersection array $\{4,2,2;1,1,2\}$ and the graph $\Gamma_{+, 7}$ is a distance-regular graph with intersection array $\{3, 2, 2, 1; 1, 1, 1, 2\}$.

		For $q = 9$, the graph $\Gamma_{+, 9}$ is a distance-regular graph with intersection array $\{4, 2, 2, 2; 1, 1, 1, 2\}$ and the graph $\Gamma_{-, 9}$ is a distance-regular graph with intersection array $\{5, 4, 2;1, 1, 4\}$.

		The graph $\Gamma_{\varepsilon, q}$ is a Deza graph with parameters
		\[
		\left(\frac{q^2 + \varepsilon q}{2}, \frac{q - \varepsilon}{2}, 1, 0\right).
		\]

		Let $n = 2m+1\ge 5$ be odd and $q = 13$. Consider $A_0 + A_3 + A_4 + A_5$.
		Let $a, b \in \{0, 3, 4, 5\}$ and $\alpha \in \widetilde{\mathbb{F}}_q$.
		By direct computation, for $\alpha\neq 1$, we have
		\[
\begin{aligned}
\sum_{a,b\in\{0,3,4,5\}}p_{a,b}^{\alpha}
&=\frac{49\cdot13^{2m-2}}{2} + \\
&\quad+\frac{\varepsilon13^{m-1}}{2}
\sum_{x,y\in\{0,\pm3,\pm4,\pm5\}}
\mu\!\left(1-\alpha^2-x^2-y^2+2\alpha xy\right),
\end{aligned}
\]
		\[
			\sum_{x,y\in\{0,\pm3,\pm4,\pm5\}}
			\mu\!\left(1-\alpha^2-x^2-y^2+2\alpha xy\right)
			=
			\begin{cases}
				-7, & \alpha\in\{3,5\},\\
				-15, & \alpha\in\{0,2,4,6\},
			\end{cases}
		\]
		\[
			\sum_{a,b\in\{0,3,4,5\}} p_{a,b}^{\alpha}
			=
			\begin{cases}
				\dfrac{49\cdot 13^{2m-2}-7\varepsilon\,13^{m-1}}{2},
				& \alpha\in\{3,5\},\\[1.2ex]
				\dfrac{49\cdot 13^{2m-2}-15\varepsilon\,13^{m-1}}{2},
				& \alpha\in\{0,2,4,6\}.
			\end{cases}
		\]

		For $\alpha=1$, the seven terms with equal signed indices contribute $13^{2m-2}-\varepsilon13^{m-1}$ each, and the remaining 42 terms contribute $13^{2m-2}$ each. Hence
		\[
		\sum_{a,b\in\{0,3,4,5\}}p_{a,b}^{1} = \frac{49\cdot13^{2m-2} - \varepsilon 7\cdot13^{m-1}}{2}.
		\]
		Thus, the graph corresponding to $A_0 + A_3 + A_4 + A_5$ is a Deza graph with parameters
		\[\begin{aligned}
			v & = \frac{13^{2m} + \varepsilon \cdot13^{m}}{2}, \\
			k & = \frac{7\cdot13^{2m-1} - \varepsilon 7\cdot13^{m-1}}{2}, \\
			b & =\frac{49\cdot13^{2m-2} + (4-11\varepsilon)13^{m-1}}{2}, \\
			a & =\frac{49\cdot13^{2m-2} - (4+11\varepsilon)13^{m-1}}{2}.
		\end{aligned}\]
		This proves statement~(4).

		Let $n = 2m+1 \ge 5$ be odd and $q \ge 5$.
		For representatives $\mathbf v,\mathbf u$ of distinct vertices, put $c=B(\mathbf v,\mathbf u)$.
		Theorem~\ref{number_of_sols_thm} gives
		\[
		2p_{0,0}(c)=
		\begin{cases}
		q\mathcal N(n-3,-\det Q,q,1), & c^2=1,\\
		\mathcal N(n-2,(1-c^2)\det Q,q,1), & c^2\ne1.
		\end{cases}
		\]
		By Lemma~\ref{number_of_sols_a}, this becomes
		\[
		2p_{0,0}(c)=
		\begin{cases}
		q^{2m-2}-\varepsilon q^{m-1}, & c^2=1,\\
		q^{2m-2}+\varepsilon\mu(c^2-1)q^{m-1}, & c^2\ne1.
		\end{cases}
		\]
		Thus the number of common neighbours takes at most two values, and $A_0$ defines a Deza graph.
		This proves statement~(3).

		Let $n=4$ and $q\ge5$, so $\varepsilon=\mu(\det Q)$.
		Again put $c=B(\mathbf v,\mathbf u)$ for representatives of distinct vertices.
		For $c^2=1$, choose signs so that $c=1$. Then
		\[
		p_{1,1}(1)=P_{1,1}(1)+P_{1,-1}(1)-2
		=q\mathcal N(1,-\det Q,q,0)+q-2=2q-2.
		\]
		For $c^2\ne1$, both $P_{1,1}(c)$ and $P_{1,-1}(c)$ count representations of nonzero elements by a non-degenerate binary form of determinant $(1-c^2)\det Q$.
		Lemma~\ref{number_of_sols_a} therefore gives
		\[
		p_{1,1}(c)=
		\begin{cases}
		2q-2, & c^2=1,\\
		2q-2\varepsilon\mu(c^2-1), & c^2\ne1.
		\end{cases}
		\]
		Hence $A_1$ defines a Deza graph. Its valency is $\mathcal N(3,\det Q,q,0)-1=q^2-1$, and its parameters are
		\[
		\left(\frac{q^3-\varepsilon q}{2},q^2-1,2q+2,2q-2\right).
		\]

		The spectrum of this graph is
		\[
		\left\{(q^2-1)^{1},\;(q-1)^\frac{(q+1)^2(q-3)}4,\;(0)^{q^2},\;(-(q+1))^\frac{(q-1)^3}{4}\right\}, \text{ if } \varepsilon = 1,
 		\]
		and
		\[
		\left\{(q^2-1)^{1},\;(q-1)^\frac{(q-1)(q^2+1)}{4},\;(-2)^{q^2},\;(-(q+1))^\frac{(q-3)(q^2+1)}{4}\right\} , \text{ if } \varepsilon = -1.
		\]
		This proves statement~(2).

		Let $n = 2m \ge 4$ be even, and $q = 9$.
		Let $\FF_9 = \FF_3(\theta)$, where $\theta^2 - \theta - 1 = 0$.
		We may choose the following representatives of $\widetilde{\mathbb{F}}_q$:
		\[
		\widetilde{\mathbb{F}}_9 = \{0, 1, \theta, \theta+1, \theta-1\}.
		\]
		By direct computation, we obtain
		\[\sum_{a, b \in \{0, \theta\}}p_{a, b}(\alpha) =
		\begin{cases}
			\frac{9^{2m-2} - \varepsilon 9^{m-1}}{2} & \alpha \in \{0, 1, \theta+1, \theta-1\}, \\
			\frac{9^{2m-2} + \varepsilon 9^{m-1}}{2} & \alpha = \theta.
		\end{cases}
		\]
		Similarly,
		\[\sum_{a, b \in \{0, \theta+1\}}p_{a, b}(\alpha) =
		\begin{cases}
			\frac{9^{2m-2} + \varepsilon 9^{m-1}}{2} & \alpha \in \{\theta, \theta-1\}, \\
			\frac{9^{2m-2} + \varepsilon 3\cdot9^{m-1}}{2} & \alpha \in \{0, 1, \theta + 1\}.
		\end{cases}
		\]
		Note that $(\theta-1)^\phi = -\theta$, where $\phi$ is the Frobenius automorphism.
		Hence \[A_0 + A_\theta \cong A_0 + A_{\theta-1}.\]

		Thus the graph corresponding to $A_0 + A_a$ for $a\in\FF_9\setminus\FF_3$ is a Deza graph. Its parameters are as follows.
		For $A_0 + A_\theta$, the parameters are
		\[
		\left(
		\frac{9^{2m-1} - \varepsilon 9^{m-1}}{2},
		\frac{3\cdot9^{2m-2} - \varepsilon 9^{m-1}}{2},
		\frac{9^{2m-2} + 9^{m-1}}{2},
		\frac{9^{2m-2} - 9^{m-1}}{2}
		\right),
		\]

		For $A_0 + A_{\theta+1}$, the parameters are
		\[
\begin{gathered}
\left(\frac{9^{2m-1}-\varepsilon9^{m-1}}{2},\;
\frac{3\cdot9^{2m-2}+3\varepsilon9^{m-1}}{2},\right.\\
\left.\frac{9^{2m-2}+(2\varepsilon+1)9^{m-1}}{2},\;
\frac{9^{2m-2}+(2\varepsilon-1)9^{m-1}}{2}\right).
\end{gathered}
\]
	\end{proof}

	\begin{rem}
    The graphs with parameters $(360,135,54,45)$ and $(369,108,36,27)$ arising
    in Theorem~\ref{thm:families}(1), the graph $(65,34,18,15)$ in part~(6),
    and the graph $(168,111,75,70)$ in part~(13) were also independently
    constructed by Smaldore~\cite{Smaldore}.
    \end{rem}

    \section{The case \texorpdfstring{$q = 5$}{q = 5}}\label{sec:q5}
	Let $n>3$ and let $\FF = \mathrm{GF}(5)$ be the finite field with $5$ elements.
	Let $V = \FF^n$ be the vector space of dimension $n$ over $\FF$.
	Let $\mathbf{0} = (0, 0, \dots, 0)$ denote the zero vector in $V$.
	For vectors $\mathbf{v}, \mathbf{u} \in V$, let $B(\mathbf{u}, \mathbf{v})$ be the bilinear form associated with the quadratic form $Q$.
	Denote
	\[
	\varepsilon =
	\begin{cases}
		\mu((-1)^\frac{n-1}{2}\det Q), & n\text{ odd,} \\
		\mu((-1)^\frac{n}{2} \det Q), & n\text{ even},
	\end{cases}
	\]
	where $\mu$ denotes the quadratic character of $\FF$.
	Recall that $S=\{\mathbf v\in V:Q(\mathbf v)=1\}$ and $S^+ = \{\{\mathbf{v}, -\mathbf{v}\} \mid \mathbf{v} \in S\}$.

	We define the following relations on $S^+$:
	\begin{align*}
		{\cal R}_\Delta &= \{(\mathbf{v}, \mathbf{v}) \,|\, \mathbf{v}\in S^+\}, \\
		{\cal R}_0 &= \{(\mathbf{v}, \mathbf{u}) \,|\, B(\mathbf{v}, \mathbf{u}) = 0\}, \\
		{\cal R}_1 &= \{(\mathbf{v}, \mathbf{u}) \,|\, \mathbf{v} \ne \mathbf{u}\,,\, B(\mathbf{v}, \mathbf{u}) = \pm 1\}, \\
		{\cal R}_2 &= \{(\mathbf{v}, \mathbf{u}) \,|\, B(\mathbf{v}, \mathbf{u}) = \pm 2\}.
	\end{align*}

	Then $\mathcal{A}_5 = (S^+, \{{\cal R}_\Delta, {\cal R}_0, {\cal R}_1, {\cal R}_2\})$ is an association scheme with $3$ classes.
	Let $A_i$ denote the adjacency matrix of the relation ${\cal R}_i$.

	\begin{thm*}
		Let $n > 3$ be odd.
		Then $A_2$ is the adjacency matrix of an edge-regular Deza graph with parameters
		\[
		\left(
		\frac{5^{n-1} + \varepsilon 5^{\frac{n-1}{2}}}{2},
		5^{n-2} - \varepsilon 5^{\frac{n-3}{2}},
		2\cdot5^{n-3} + \frac{1-3\varepsilon}{2}5^{\frac{n-3}{2}},
		2\cdot5^{n-3} - \frac{1+3\varepsilon}{2}5^{\frac{n-3}{2}}
		\right).
		\]

		Moreover, $A_0$ and $A_1 + A_2$ are adjacency matrices of complementary strongly regular graphs.
		The graph corresponding to $A_0$ is the child associated with $a$ when $\varepsilon=1$ and with $b$ when $\varepsilon=-1$; the other child corresponds to $A_1+A_2$. The graph with adjacency matrix $A_0$ has parameters
		\[
		\left(
		\frac{5^{n-1} + \varepsilon 5^{\frac{n-1}{2}}}{2},
		\frac{5^{n-2} - \varepsilon 5^{\frac{n-3}{2}}}{2},
		\frac{5^{n-3} + \varepsilon 5^\frac{n-3}{2}}{2},
		\frac{5^{n-3} - \varepsilon 5^\frac{n-3}{2}}{2}
		\right).
		\]

	\end{thm*}

	By Lemma~\ref{number_of_sols_a}, we have
	\[
	|\{\mathbf{v} \in \FF_q^n \,|\, Q(\mathbf{v}) = 1\}|
	=
	\begin{cases}
		5^{n-1} + \varepsilon 5^\frac{n-1}{2}, & n\text{ odd}, \\
		5^{n-1} - \varepsilon 5^{\frac{n}{2}-1}, & n\text{ even}.
	\end{cases}
	\]
	Thus,
	\[
	|S^+|
	= \frac{1}{2}
	\begin{cases}
		5^{n-1} + \varepsilon 5^{\frac{n-1}{2}}, & n\text{ odd}, \\
		5^{n-1} - \varepsilon 5^{\frac{n}{2}-1}, & n\text{ even}.
	\end{cases}
	\]

	\textbf{Valencies.}
	For $\alpha\in\{0,1,2\}$ and a fixed $\mathbf v\in S$, we have
	\[
	k_\alpha = p_{\alpha,\alpha}^{\Delta}
	=\frac12\left|\{\mathbf x\in S: B(\mathbf v,\mathbf x)=\pm\alpha\}
	\setminus\{\mathbf v,-\mathbf v\}\right|.
	\]
	Hence
	\[
	k_\alpha = \begin{cases}
		1 & \alpha = \Delta, \\
		\frac{K_{\alpha}}{2}  & \alpha = 0, \\
		K_{\alpha} - 1  & \alpha = 1, \\
		K_{\alpha} & \text{ otherwise},
	\end{cases}
	\]
	where for $a \in \FF$, $K_{a}$ is the number of solutions of the system
	\[
	\begin{cases}
		B(\mathbf{x}, \mathbf{x}) = 1, \\
		B(\mathbf{v}, \mathbf{x}) = a, 
\end{cases}
	\]
	and by Theorem~\ref{number_of_sols_K_thm} we have $K_{a} = \mathcal{N}(n-1, \det Q, q, 1-a^2)$.

	For the remainder of the paper, let $n = 2m+1$ be odd. Then
	\begin{align*}
		|S^+|
		&= \frac{5^{n-1} + \varepsilon 5^{\frac{n-1}{2}}}{2},\\
		k_\Delta   &= 1,\\
		k_0   &= \frac{\mathcal{N}(n-1, \det Q, 5, 1)}{2} = \frac{1}{2} \left( 5^{n-2} - \varepsilon 5^{\frac{n-3}{2}} \right),\\
		k_1   &= \mathcal{N}(n-1, \det Q, 5, 0) - 1 = 5^{n-2} + \varepsilon 4 \cdot 5^{\frac{n-3}{2}} - 1,\\
		k_2   &= \mathcal{N}(n-1, \det Q, 5, 2) = 5^{n-2} - \varepsilon 5^{\frac{n-3}{2}}.
\end{align*}

	Let $\mathbf{v}, \mathbf{u} \in S$ be representatives of distinct elements of $S^+$.
	Let $a, b \in \FF_q$. Using Theorem~\ref{number_of_sols_thm} and Lemmas~\ref{number_of_sols_0}--\ref{number_of_sols_a}, we compute the number of solutions to system~\eqref{Pvu-system} when $Q(\mathbf{v}) = Q(\mathbf{u}) = 1$ and $B(\mathbf{v}, \mathbf{u}) \in \{0, 1, 2\}$.

	The counts satisfy $P_{a,b}(c)=P_{b,a}(c)=P_{-a,-b}(c)$.
	In the formulas below, pairs obtained by these symmetries are included implicitly.

	\textbf{Case $B(\mathbf{v}, \mathbf{u}) = 0$:}

	By Theorem~\ref{number_of_sols_thm} we have
	\[
	P_{a, b}(0) = \mathcal{N}(n-2, \det Q, 5, 1 - (a^2+b^2)).
	\]

	Hence
	\[
	P_{a, b}(0) =
	\begin{cases}
		5^{n-3} + \varepsilon 5^\frac{n-3}{2},   & (a,b) \in \{(0, 0), (\pm 1, \pm 1), (\pm 1, \pm 2)\}, \\
		5^{n-3},				       & (a,b) \in \{(0, \pm 1)\}, \\
		5^{n-3} - \varepsilon 5^{\frac{n-3}{2}}, & (a,b) \in \{(0, \pm 2), (\pm 2, \pm 2)\}.
	\end{cases}
	\]

	\textbf{Case $B(\mathbf{v}, \mathbf{u}) = 1$:}

	By Theorem~\ref{number_of_sols_thm} we have
	\[
	P_{a, b}(1) =
	\begin{cases}
		5^{n-3} + \varepsilon 4 \cdot 5^\frac{n-3}{2}, & a=b\in\{1,-1\}, \\
		5^{n-3} - \varepsilon 5^{\frac{n-3}{2}},       & a=b\in\{0,2,-2\}, \\
		5^{n-3},                           & \text{ if } a \ne b.
	\end{cases}
	\]

	\textbf{Case $B(\mathbf{v}, \mathbf{u}) = 2$:}

	By Theorem~\ref{number_of_sols_thm} we have
	\[
	P_{a, b}(2) = 5^{n-3} + \varepsilon \mu\left(2 - a^2 - ab - b^2 \right) 5^{\frac{n-3}{2}}.
	\]
	Hence
	\[
	P_{a, b}(2) =
	\begin{cases}
		5^{n-3}                     & (a,b) \in \{(1, 2), (-1, -2), (2, 2), (-2, -2)\}, \\
		5^{n-3} + \varepsilon 5^{\frac{n-3}{2}} & (a,b) \in \{(0, \pm 1), (\pm 1, \pm 1), (1, -2), (-1, 2)\}, \\
		5^{n-3} - \varepsilon 5^{\frac{n-3}{2}} & (a,b) \in \{(0,0), (0, \pm 2), (2, -2)\}.
	\end{cases}
	\]

	The number $p_{i,j}^k$ is the intersection number for $i,j,k \in \{0,1,2\}$.
	We use the formula
	\[
	|p_{a,b}(\mathbf{v}, \mathbf{u})| = p_{a, b}^{\alpha}
	=
	\frac12\sum_{(x,y)\in S(a,b)} P_{x,y}(\alpha)
	\;-\;\delta_{a,1}\,\delta_{\alpha,b}\;-\;\delta_{b,1}\,\delta_{\alpha,a}.
	\]
	Here $(\mathbf{v}, \mathbf{u}) \in {\cal R}_\alpha$,
	\[
	S(a,b)=\{(\varepsilon_1 a,\varepsilon_2 b)\;:\;\varepsilon_1,\varepsilon_2\in\{-1, 1\}\},
	\]
	and $\delta_{x,y}$ is the Kronecker delta.
	This gives the intersection numbers of the scheme $\mathcal{A}_5$.
	Let $B_k=(p_{k,j}^{i})_{i,j\in\{\Delta,0,1,2\}}$ denote the intersection matrix for the relation ${\cal R}_k$, with rows and columns ordered as $(\Delta,0,1,2)$.
	Then

	\[
	B_\Delta =
	\begin{pmatrix}
		1 & 0 & 0 & 0 \\
		0 & 1 & 0 & 0 \\
		0 & 0 & 1 & 0 \\
		0 & 0 & 0 & 1
	\end{pmatrix}
	\]

	\[
	B_0 =
	\begin{pmatrix}
		0 & k_0 & 0 & 0 \\
		1 & \frac{P_{0,0}(0)}{2} & P_{1,0}(0) - 1 & P_{2,0}(0) \\
		0 & \frac{P_{0,0}(1)}{2} & P_{1,0}(1) & P_{2,0}(1) \\
		0 & \frac{P_{0,0}(2)}{2} & P_{1,0}(2) & P_{2,0}(2)
	\end{pmatrix}
	\]
	\[
	B_1 =
	\begin{pmatrix}
		0 & 0 & k_1 & 0 \\
		0 & P_{0,1}(0) - 1 & 2P_{1,1}(0) & 2P_{1,2}(0) \\
		1 & P_{0,1}(1) & P_{1,1}(1) + P_{-1,1}(1) - 2 & 2P_{1,2}(1) \\
		0 & P_{0,1}(2) & 2P_{1,1}(2) & P_{1,2}(2) + P_{-1,2}(2) - 1
	\end{pmatrix}
	\]
	\[
	B_2 =
	\begin{pmatrix}
		0
		& 0 & 0 & k_2 \\
		0 & P_{0,2}(0) & 2P_{1,2}(0) & 2P_{2,2}(0) \\
		0 & P_{0,2}(1) & 2P_{1,2}(1) & P_{2,2}(1) + P_{-2,2}(1) \\
		1 & P_{0,2}(2) & P_{1,2}(2) + P_{-1,2}(2) - 1 & P_{2,2}(2) + P_{2,-2}(2)
	\end{pmatrix}
	\]

	Let $h = 5^\frac{n-3}{2}$. The intersection numbers of the scheme $\mathcal{A}_5$ are
	\begin{alignat*}{3}
		p_{0,0}^0 &= \tfrac{1}{2} h^{2} +\tfrac{1}{2} {\varepsilon} h,
		&\quad
		p_{0,0}^1 &= \tfrac{1}{2} h^{2} - \tfrac{1}{2} {\varepsilon} h,
		&\quad
		p_{0,0}^2 &= \tfrac{1}{2} h^{2} - \tfrac{1}{2} {\varepsilon} h, \\
		p_{0,1}^0 &= h^{2} - 1,
		&\quad
		p_{0,1}^1 &= h^{2},
		&\quad
		p_{0,1}^2 &= h^{2} + {\varepsilon} h, \\
		p_{0,2}^0 &= h^{2} - {\varepsilon} h,
		&\quad
		p_{0,2}^1 &= h^{2},
		&\quad
		p_{0,2}^2 &= h^{2} - {\varepsilon} h, \\
		p_{1,0}^0 &= h^{2} - 1,
		&\quad
		p_{1,0}^1 &= h^{2},
		&\quad
		p_{1,0}^2 &= h^{2} +
		{\varepsilon} h, \\
		p_{1,1}^0 &= 2 h^{2} + 2 {\varepsilon} h,
		&\quad
		p_{1,1}^1 &= 2 h^{2} + 4 {\varepsilon} h - 2,
		&\quad
		p_{1,1}^2 &= 2 h^{2} + 2 {\varepsilon} h, \\
		p_{1,2}^0 &= 2 h^{2} + 2 {\varepsilon} h,
		&\quad
		p_{1,2}^1 &= 2 h^{2},
		&\quad
		p_{1,2}^2 &= 2 h^{2} + {\varepsilon} h - 1, \\
		p_{2,0}^0 &= h^{2} - {\varepsilon} h,
		&\quad
		p_{2,0}^1 &= h^{2},
		&\quad
		p_{2,0}^2 &= h^{2} - {\varepsilon} h, \\
		p_{2,1}^0 &= 2 h^{2} + 2 {\varepsilon} h,
		&\quad
		p_{2,1}^1 &= 2 h^{2},
		&\quad
		p_{2,1}^2 &= 2 h^{2} + {\varepsilon} h - 1, \\
		p_{2,2}^0 &= 2 h^{2} - 2 {\varepsilon} h,
		&\quad
		p_{2,2}^1 &= 2 h^{2} - {\varepsilon} h,
		&\quad
		p_{2,2}^2 &= 2 h^{2} - {\varepsilon} h.
	\end{alignat*}

	The first eigenmatrices are
	\[
	P = \begin{pmatrix}
		1 & \tfrac{5}{2} h^{2} - \tfrac{1}{2} h & 5 h^{2} + 4 h - 1 & 5 h^{2} - h \\
		1 & 2h & -h - 1 & -h \\
		1 & -h & 3h - 1 & -2h \\
		1 & -h & -h - 1 & 2h
	\end{pmatrix}, \quad \text{ if } \varepsilon = 1,
	\]
	\[
	P = \begin{pmatrix}
		1 & \tfrac{5}{2} h^{2} + \tfrac{1}{2} h & 5 h^{2} - 4 h - 1 & 5 h^{2} + h \\
		1 & h & h - 1 & -2h \\
		1 & h & -3h - 1 & 2h \\
		1 &
		-2h & h - 1 & h
	\end{pmatrix}, \quad \text{ if } \varepsilon = -1.
	\]

	\subsection{Proof of the main theorem}
	\begin{lem}
		$p^1_{2,2} = p^2_{2,2} = 2\cdot5^{n-3} - \varepsilon 5^{\frac{n-3}{2}}.$
	\end{lem}
	\begin{proof}
		We need to show that
		$P_{2,2}(1) + P_{-2,2}(1) = P_{2,2}(2) + P_{2,-2}(2) = 2\cdot5^{n-3} - \varepsilon 5^{\frac{n-3}{2}}$.
		Indeed,
		\begin{align*}
			P_{2,2}(1) &= 5^{n-3} - \varepsilon 5^{\frac{n-3}{2}}, \\
			P_{-2,2}(1) &= 5^{n-3}, \\
			P_{2,2}(2) &= 5^{n-3}, \\
			P_{2,-2}(2) &= 5^{n-3} - \varepsilon 5^{\frac{n-3}{2}}.
		\end{align*}
	\end{proof}

	\begin{proof}[Proof of the main theorem]
	For $i\in\{0,1,2\}$, let $A_i$ denote the adjacency matrix of the relation ${\cal R}_i$.
	Let $I$ denote the identity matrix of appropriate size.

	Then
	\begin{align*}
		A_2^2 &= k_2I + p_{2,2}^0A_0 + p_{2,2}^1A_1 + p_{2,2}^2A_2 \\
		&= k_2I + 2P_{2,2}(0)A_0 + (P_{2,2}(1) + P_{-2,2}(1))A_1 \\
		&\quad + (P_{2,2}(2) + P_{-2,2}(2))A_2 \\
		&= k_2I + 2(5^{n-3} - \varepsilon 5^{\frac{n-3}{2}})A_0 + (2\cdot5^{n-3} - \varepsilon 5^{\frac{n-3}{2}})(A_1 + A_2).
	\end{align*}
	All adjacent pairs are in $\mathcal R_2$ and have $p_{2,2}^2$ common neighbours, so this graph is edge-regular.
	Thus $A_2$ is the adjacency matrix of a Deza graph with parameters
	\[
	\left(
	\frac{5^{n-1} + \varepsilon 5^{\frac{n-1}{2}}}{2},
	5^{n-2} - \varepsilon 5^{\frac{n-3}{2}},
	2\cdot5^{n-3} + \frac{1-3\varepsilon}{2}5^{\frac{n-3}{2}},
		2\cdot5^{n-3} - \frac{1+3\varepsilon}{2}5^{\frac{n-3}{2}}
	\right).
	\]

	To show that $A_0$ and $A_1 + A_2$ are adjacency matrices of complementary strongly regular graphs, we compute
	\begin{align*}
		A_0^2 &= k_0I + p_{0,0}^0A_0 + p_{0,0}^1A_1 + p_{0,0}^2A_2 \\
		&= k_0I + \frac{P_{0,0}(0)}{2}A_0 +  \frac{P_{0,0}(1)}{2}A_1 +  \frac{P_{0,0}(2)}{2}A_2 \\
		&= \frac{5^{n-2} - \varepsilon 5^{\frac{n-3}{2}}}{2} I + \frac{5^{n-3} + \varepsilon 5^\frac{n-3}{2}}{2}A_0 +  \frac{5^{n-3} - \varepsilon 5^\frac{n-3}{2}}{2}(A_1 + A_2).
	\end{align*}

	Thus $A_0$ is the adjacency matrix of a strongly regular graph with parameters
	\[
	\left(
	\frac{5^{n-1} + \varepsilon 5^{\frac{n-1}{2}}}{2},
	\frac{5^{n-2} - \varepsilon 5^{\frac{n-3}{2}}}{2},
	\frac{5^{n-3} + \varepsilon 5^\frac{n-3}{2}}{2},
	\frac{5^{n-3} - \varepsilon 5^\frac{n-3}{2}}{2}
	\right).
	\]
	These parameters coincide with those of the graph $NO^{\varepsilon\perp}_n(5)$.
	\end{proof}

	\backmatter
	
\section*{Statements and Declarations}
\subsection*{Funding}
This work was carried out at the Ural Mathematical Center with financial support from the Ministry of Science and Higher Education of the Russian Federation (Agreement No.~075-02-2026-737).

\end{document}